\documentclass[12pt]{amsart}

\newtheorem{theorem}{Theorem}
\newtheorem{lemma}[theorem]{Lemma}
\newtheorem{conjecture}[theorem]{Conjecture}
\theoremstyle{remark}

\begin{document}

\title[Finite type of a shrinking Ricci soliton]
 {Complete gradient shrinking Ricci solitons have finite topological type}

\author[F. Fang]{Fuquan Fang}
\thanks{Supported by NSF Grant of China \#10671097 and the Capital Normal University}
\address{Department of Mathematics, Capital Normal University,
Beijing, P.R. China}
\email{fuquan\_fang@yahoo.com}
\author[J. Man]{Jianwen Man}
\address{Nankai Institute of Mathematics,
Weijin Road 94, Tianjin 300071, P.R. China}
\author[Z. Zhang]{Zhenlei Zhang}
\address{Nankai Institute of Mathematics,
Weijin Road 94, Tianjin 300071, P.R. China}
\email{zhleigo@yahoo.com.cn}



\date{}

\dedicatory{}

\commby{}

\maketitle

%----------------------------------------------------------------------

\begin{abstract}
We show that a complete Riemannian manifold has finite topological
type (i.e., homeomorphic  to the interior of a compact manifold
with boundary), provided its Bakry-\'{E}mery Ricci tensor has a
positive lower bound, and either of the following conditions:
\begin{itemize}
 \item[(i)] the Ricci curvature is bounded from above;
 \item[(ii)] the Ricci curvature is bounded from below and injectivity radius is bounded away from
 zero.
\end{itemize}
 Moreover, a
complete shrinking Ricci soliton has finite topological type if
its scalar curvature is bounded.
\end{abstract}

%-------------------------------------------------------------------------------

\section{Introduction}

In 1968, J. Milnor \cite{Mi} conjectured that a complete
non-compact Riemannian manifold with non-negative Ricci curvature
has a finitely generated fundamental group. However, such a
manifold may not have finite topological type. Examples of
complete non-compact manifold with positive Ricci curvature
without finite topological type was constructed by Gromoll-Meyer
\cite{GM}. It has been an interesting topic in Riemannian geometry
to study the topology of complete manifolds with positive
(non-negative) Ricci curvature.

In this note we are concerned with complete Riemannian manifold
$(M,g)$ satisfying that $\text{Ric}+\text{Hess}(f)\geq\lambda g$
for some constant $\lambda>0$ and $f\in C^{\infty}(M)$, i.e.,
whose Bakry-\'{E}mery Ricci tensor is bounded below by $\lambda$
in the sense of \cite{Lo}. When the equality holds, the manifold
is a shrinking Ricci soliton, i.e., a self-similar solution of the
well-known  Ricci flow equation. If $f$ is constant,
Bakry-\'{E}mery Ricci tensor reduces to the Ricci tensor, and so
the classical Myers' theorem implies that $M$ is compact with
finite fundamental group. In general, $M$ may not be compact, but
from the work of \cite{FG,LX,Lo,W,WW,Zh} etc., $M$ still has
finite fundamental group.

The main result of this note shows that a complete Riemannian
manifold whose Bakry-\'{E}mery Ricci tensor is bounded below by
$\lambda >0$ has finite topological type, provided the Ricci
curvature is bounded from above. Moreover, a shrinking Ricci soliton
has finite topological type if its scalar curvature is bounded.

\begin{theorem}\label{T}
Suppose $(M,g)$ is a complete Riemannian manifold  satisfying
$\text{Ric}+\text{Hess}(f)\geq\lambda g$ for some constant
$\lambda>0$ and $f\in C^{\infty}(M)$. Then $M$ is of finite
topological type, if either of the following alternative conditions
holds:
\begin{itemize}
 \item[(i)] $\text{Ric}\leq Cg$ for some constant $C<\infty$;
 \item[(ii)] $\text{Ric}\geq-\delta^{-1}g$ and the injectivity radius
 ${\rm inj}(M,g)\geq\delta>0$ for some
 $\delta>0$.
\end{itemize}
\end{theorem}

If $(M,g)$ is a shrinking Ricci soliton, then the Ricci curvature
bounds can be relaxed by scalar curvature.

\begin{theorem}\label{TT}
Suppose $(M,g)$ is a complete shrinking Ricci soliton
$Ric+Hess(f)=\frac{g}{2}$, where $f\in C^{\infty}(M)$. If the
scalar curvature $R$ is bounded, then $M$ has finite topological
type.
\end{theorem}

In view of Theorem 2 it is nature to pose the following

\begin{conjecture}
Any shrinking Ricci soliton has finite topological type.
\end{conjecture}

We prove Theorem 1 in section 2 and Theorem 2 in section 3.

%---------------------------------------------------------------------------------------------------

%-------------------------------------------------------------------------------------------------

\section{Proof of theorem \ref{T}}

Let $(M,g)$ be such a manifold satisfying that
$\text{Ric}+\text{Hess}(f)\geq\lambda g$ for some $\lambda>0$ and
$f\in C^{\infty}(M)$. By the deformation lemma of Morse theory, to
prove Theorem \ref{T}, it suffices to show that the function $f$ is
proper and has no critical points outside of a compact set.

First fix one point $p\in M$ as a base point. For any $q\in M$ with
$d(p,q)=L$, choose a shortest geodesic $\gamma$ from $p$ to $q$
parametrized by arc length. Then
\begin{eqnarray}
\langle \nabla f,\dot{\gamma}\rangle (q)&=&\langle\nabla
f,\dot{\gamma}\rangle(p)+\int_{0}^{L}\frac{d^{2}}{dt^{2}}f(\gamma(t))dt\nonumber\\
&\geq&\langle\nabla
f,\dot{\gamma}\rangle(p)+\int_{0}^{L}(\lambda-Ric(\dot{\gamma},\dot{\gamma}))dt\nonumber\\
&\geq&\lambda L-|\nabla
f|(p)-\int_{0}^{L}Ric(\dot{\gamma},\dot{\gamma})dt.\nonumber
\end{eqnarray}

If the integral $$\int_{0}^{L}Ric(\dot{\gamma},\dot{\gamma})dt\leq
\Lambda$$ for some constant $\Lambda$ independent of $q$ and the
choice of $\gamma$, then $$|\nabla f|(q)\geq\langle\nabla
f,\dot{\gamma}\rangle(q)\geq\lambda d(p,q)-|\nabla f|(p)-\Lambda,$$
which implies that $|\nabla f|(q)$ has a linear growth in $d(p,q)$
and so $f$ is a proper function without critical points outside of a
compact set. In the remainder of this section, we will focus on
proving that $\int_{0}^{L}Ric(\dot{\gamma},\dot{\gamma})$ has an
upper bound under the assumptions of Theorem \ref{T}.

\vskip 2mm

Case (i): $\text{Ric}\leq Cg$ for some constant $C<\infty$;

By Lemma 2.2 of \cite{W}, the integral bound is given by
$\Lambda=2(n-1)+2C$.

\vskip 2mm

Case (ii): $\text{Ric}\geq-\delta^{-1}g$ and   ${\rm
inj}(M,g)\geq\delta>0$ for some $\delta>0$.

Suppose $d(p,q)=L\geq\delta$. Let
$\varphi(t):[0,L]\rightarrow[0,1]$ be an arcwise smooth function
such that $\varphi(0)=\varphi(L)=0$. By the second variation
formula, as did in \cite{W} or \cite{Zh}, we have the following
estimate:
$$\int_{0}^{L}\varphi^{2}(t)Ric(\dot{\gamma},\dot{\gamma})dt\leq(n-1)\int_{0}^{L}|\dot{\varphi}|^{2}dt.$$
Now define $\varphi$ by
\begin{equation}\nonumber
\varphi=\left\{ \begin{array}{ll}
        \frac{3}{\delta}t,&t\in[0,\frac{\delta}{3}]; \\
        1,&t\in[\frac{\delta}{3},L-\frac{\delta}{3}];\\
        \frac{3}{\delta}(L-t),&t\in[L-\frac{\delta}{3},L],
\end{array} \right.
\end{equation}
then we have the estimate
\begin{eqnarray}
\int_{0}^{L}Ric(\dot{\gamma},\dot{\gamma})dt&\leq&(n-1)\int_{0}^{L}|\dot{\varphi}|^{2}dt
+\int_{0}^{\frac{\delta}{3}}(1-\varphi^{2})Ric(\dot{\gamma},\dot{\gamma})dt
+\int_{L-\frac{\delta}{3}}^{L}(1-\varphi^{2})Ric(\dot{\gamma},\dot{\gamma})dt\nonumber\\
&\leq&\frac{6}{\delta}(n-1)+\frac{2}{3}+\int_{0}^{\frac{\delta}{3}}Ric(\dot{\gamma},\dot{\gamma})dt+
\int_{L-\frac{\delta}{3}}^{L}Ric(\dot{\gamma},\dot{\gamma})dt\nonumber,
\end{eqnarray}
where in the second inequality, we used the fact that
\begin{eqnarray}
\int_{0}^{\frac{\delta}{3}}(1-\varphi^{2})Ric(\dot{\gamma},\dot{\gamma})dt&=&
\int_{0}^{\frac{\delta}{3}}(1-\varphi^{2})(Ric(\dot{\gamma},\dot{\gamma})+\frac{1}{\delta})dt
-\frac{1}{\delta}\int_{0}^{\frac{\delta}{3}}(1-\varphi^{2})dt\nonumber\\
&\leq&\int_{0}^{\frac{\delta}{3}}(Ric(\dot{\gamma},\dot{\gamma})+\frac{1}{\delta})dt
-\frac{1}{\delta}\int_{0}^{\frac{\delta}{3}}(1-\varphi^{2})dt\nonumber\\
&\leq&\int_{0}^{\frac{\delta}{3}}Ric(\dot{\gamma},\dot{\gamma})dt+
\frac{1}{\delta}\int_{0}^{\frac{\delta}{3}}\varphi^{2}dt\nonumber\\
&\leq&\int_{0}^{\frac{\delta}{3}}Ric(\dot{\gamma},\dot{\gamma})dt+\frac{1}{3},\nonumber
\end{eqnarray}
and similarly
$$\int_{L-\frac{\delta}{3}}^{L}(1-\varphi^{2})Ric(\dot{\gamma},\dot{\gamma})dt\leq
\int_{L-\frac{\delta}{3}}^{L}Ric(\dot{\gamma},\dot{\gamma})dt+\frac{\delta}{3}.$$

We next prove that
$\int_{0}^{\frac{\delta}{3}}Ric(\dot{\gamma},\dot{\gamma})dt$ and
$\int_{L-\frac{\delta}{3}}^{L}Ric(\dot{\gamma},\dot{\gamma})dt$ are
bounded from above and so finish the proof of Theorem 1. This is
given by the following lemma.

\begin{lemma}
If $\text{Ric}\geq-\delta^{-1}g$ and   ${\rm
inj}(M,g)\geq\delta>0$ for some $\delta>0$, then
$$\int_{0}^{\frac{\delta}{3}}Ric(\dot{\gamma},\dot{\gamma})dt\leq\frac{6}{\delta}(n-1)+\frac{2}{3}$$
for any minimal arc length parametrized geodesic
$\gamma:[0,\frac{\delta}{3}]\rightarrow M$.
\end{lemma}
\begin{proof}
Firstly, by ${\rm inj}(M,g)\geq\delta$, we can extend the geodesic
$\gamma$ to a shortest geodesic $\sigma:[0,\delta]\rightarrow M$,
such that
$\gamma(t)=\sigma(t+\frac{\delta}{3}),t\in[0,\frac{\delta}{3}]$.

Set $L=\delta$ in the arguments above, we have
$$\int_{0}^{\delta}\varphi^{2}Ric(\dot{\sigma},\dot{\sigma})dt\leq(n-1)\int_{0}^{\delta}|\dot{\varphi}|^{2}dt
=\frac{6}{\delta}(n-1),$$ then using $Ric\geq-\delta^{-1}g$, we get
the estimate
\begin{eqnarray}
\int_{0}^{\frac{\delta}{3}}Ric(\dot{\gamma},\dot{\gamma})dt&=&\int_{\frac{\delta}{3}}^{\frac{2\delta}{3}}
Ric(\dot{\sigma},\dot{\sigma})dt\nonumber\\
&\leq&\frac{6}{\delta}(n-1)-\int_{0}^{\frac{\delta}{3}}\varphi^{2}Ric(\dot{\sigma},\dot{\sigma})dt
-\int_{\frac{2\delta}{3}}^{\delta}\varphi^{2}Ric(\dot{\sigma},\dot{\sigma})dt\nonumber\\
&\leq&\frac{6}{\delta}(n-1)+\frac{2}{3}.\nonumber
\end{eqnarray}
This concludes the result.
\end{proof}

%------------------------------------------------------------------------------------------------------------------

%----------------------------------------------------------------------------------------------------------------

\section{Proof of Theorem \ref{TT}}

As before, we will prove that the potential function $f$ to the
Ricci soliton is proper and has no critical points outside of
$B(p,\rho)$ for large $\rho$.

Suppose $(M,g)$ is a complete shrinking Ricci soliton which
satisfies
\begin{equation}\label{e1}
\text{Ric}+\text{Hess}(f)=\frac{g}{2}
\end{equation}
for some potential function $f$. Suppose further that the scalar
curvature $|R|\leq C$ for some constant $C<\infty$. It's well-known
that the following analytic equality holds for the soliton ( after
modifying $f$ by a translation, see \cite{CHI} for example.):
\begin{equation}\label{e2}
R+|\nabla f|^{2}=f.
\end{equation}

We begin with several lemmas. Let $q\in M$ be one critical point of
$f$ and denote by $\rho=d(p,q)$ the distance from $p$ to $q$. Let
$\gamma$ be a shortest arc length parametrized geodesic from $p$ to
$q$. Then we have
\begin{lemma}\label{l21}
\begin{equation}\label{e3}
\frac{\rho}{2}-|\nabla
f|(p)\leq\int_{0}^{\rho}Ric(\dot{\gamma},\dot{\gamma})dt.
\end{equation}
\end{lemma}
\begin{proof}
By a direct computation,
\begin{eqnarray}
0=\langle \nabla f,\dot{\gamma}\rangle(q)&=&\langle\nabla
f,\dot{\gamma}\rangle (p)+\int_{0}^{\rho}\frac{d^{2}}{dt^{2}}f(\gamma(t))dt\nonumber\\
&\geq&-|\nabla
f|(p)+\int_{0}^{\rho}(\frac{1}{2}-Ric(\dot{\gamma},\dot{\gamma}))dt.\nonumber
\end{eqnarray}
Then the result follows.
\end{proof}

On the other hand, by second variation formula as did in above
section, we can get an upper bound for
$\int_{0}^{\rho}Ric(\dot{\gamma},\dot{\gamma})dt.$ Precisely, for
the function $\psi$ defined by
$$\psi(t)=t,t\in[0,1];\psi(t)\equiv1,t\in[1,\rho_{i}-1];\psi(t)=\rho_{i}-t,t\in[\rho_{i}-1,\rho_{i}],$$
we have the estimate
\begin{eqnarray}
\int_{0}^{\rho}Ric(\dot{\gamma},\dot{\gamma})dt&\leq&\int_{0}^{\rho}(n-1)|\dot{\psi}|^{2}dt+
\int_{0}^{1}(1-\psi^{2})Ric(\dot{\gamma},\dot{\gamma})dt\nonumber\\
&+&
\int_{\rho-1}^{\rho}(1-\psi^{2})Ric(\dot{\gamma},\dot{\gamma})dt\nonumber\\
&\leq&2(n-1)+\sup_{B(p,1)}|Ric|+\int_{\rho-1}^{\rho}(1-\psi^{2})(\frac{1}{2}-\frac{d^{2}}{dt^{2}}f(\gamma(t)))dt\nonumber\\
&\leq&2(n-1)+1+\sup_{B(p,1)}|Ric|-\int_{\rho-1}^{\rho}(1-\psi^{2})\frac{d^{2}}{dt^{2}}f(\gamma(t))dt\nonumber.
\end{eqnarray}

Do integration by parts, we have the estimate for the last term
\begin{eqnarray}
-\int_{\rho-1}^{\rho}(1-\psi^{2})\frac{d^{2}}{dt^{2}}f(\gamma(t))dt&=&
2\int_{\rho-1}^{\rho}\psi\frac{d}{dt}f(\gamma(t))dt\nonumber\\
&=&-2f(\gamma(\rho-1))+2\int_{\rho-1}^{\rho}f(\gamma(t))dt.\nonumber
\end{eqnarray}
Substituting this equality into above estimate, we obtain

\begin{lemma}\label{l22}
\begin{eqnarray}\label{e4}
\int_{0}^{\rho}Ric(\dot{\gamma},\dot{\gamma})dt&\leq&2n+\sup_{B(p,1)}|Ric|+2\int_{\rho-1}^{\rho}f(\gamma(t))dt
-2f(\gamma(\rho-1))\nonumber\\
&\leq&2n+\sup_{B(p,1)}|Ric|+\sup_{x,y\in
B(q,1)}2|f(x)-f(y)|.\nonumber
\end{eqnarray}
\end{lemma}

Now we use equation (\ref{e2}) to give an upper bound of
$\int_{0}^{\rho}Ric(\dot{\gamma},\dot{\gamma})dt$. First by equation
(\ref{e2}) we have the gradient estimate $|\nabla
f|\leq\sqrt{f-R}\leq\sqrt{f+C}$. Then by assumption, $q$ is a
critical point of $f$, so $|f(q)|=|R(q)|\leq C$. Integrating along a
geodesic, we see that for any $x\in B(q,1)$
\begin{equation}\label{e6}
\sqrt{f(x)+C}\leq\sqrt{f(q)+C}+\frac{d(x,q)}{2}\leq\sqrt{2C}+1.
\end{equation}
Thus $f(x)\leq3C+2$ for all $x\in B(q,1)$ and consequently
\begin{equation}\label{e7}
\sup_{x,y\in B(q,1)}|f(x)-f(y)|\leq(3C+2+C)=4C+2.
\end{equation}
The combination of Lemma \ref{l21}, Lemma \ref{l22} and equation
(\ref{e7}) gives the upper bound of the distance $\rho=d(p,q)$:
$$\rho\leq4n+8C+4+2|\nabla f|(p)+\sup_{B(p,1)}2|Ric|.$$

Note that the arguments above just used the upper boundedness of
$f$. By the same reason as before, we conclude that $f$ is proper
and then finish the proof of the theorem.

\end{document}